\documentclass[12pt]{amsart}
\usepackage{amsfonts}
\usepackage{amssymb,mathrsfs}
\usepackage{amsmath}
\usepackage{enumerate}
\usepackage{latexsym}
\usepackage{graphicx}
\usepackage{bm}
\usepackage{subfigure}
 \usepackage{float}
\usepackage{amsmath}
\usepackage{amsthm}
\usepackage{verbatim}
\usepackage{vmargin}
\usepackage{amstext}
\usepackage{array}
\usepackage{booktabs}
\usepackage{indentfirst}
\usepackage[compress]{cite}
\usepackage{url}

\usepackage{geometry}

\numberwithin{equation}{section}

\newtheorem{thm}{Theorem}[section]
\newtheorem{cor}[thm]{Corollary}
\newtheorem{lma}[thm]{Lemma}

\newtheorem{defn}[thm]{Definition}
\newtheorem{rem}[thm]{Remark}

\theoremstyle{definition}

\begin{document}
\setlength{\oddsidemargin}{80pt}
\setlength{\evensidemargin}{80pt}
\setlength{\topmargin}{80pt}
\author{Jun~Xian, Xiaoda~Xu}

\address{J.~Xian\\School of Mathematics and Guangdong Province Key Laboratory of Computational Science
 \\Sun
Yat-sen University
 \\ 
510275 Guangzhou\\
China.} \email{xianjun@mail.sysu.edu.cn}

\address{X.~Xu\\School of Mathematics
 \\Sun
Yat-sen University
 \\
510275 Guangzhou\\
China.} \email{xuxd26@mail2.sysu.edu.cn}

\title[Expected Uniform Integration approximation]{Expected uniform integration approximation under general equal measure partition}


\date{September 28, 2021}

\begin{abstract}
In this paper, we study bounds of expected $L_2-$discrepancy to give mean square error of uniform integration approximation for functions in Sobolev space $\mathcal{H}^{\mathbf{1}}(K)$, where $\mathcal{H}$ is a reproducing Hilbert space with kernel $K$. Better order $O(N^{-1-\frac{1}{d}})$ of approximation error is obtained, comparing with previously known rate $O(N^{-1})$ using crude Monte Carlo method. Secondly, we use expected $L_{p}-$discrepancy bound($p\ge 1$) of stratified samples to give several upper bounds of $p$-moment of integral approximation error in general Sobolev space $F_{d,q}^{*}$.
\end{abstract}

\maketitle

\section{Introduction}\label{intro}

Let $f: [0,1]^{d}\rightarrow \mathbb{R}$, a way of approximating the integral $$I(f)=\int_{[0,1]^{d}}f(x)dx$$ consists of randomly drawing points $\mathbf{P}=\{x_1,x_2,\ldots,x_N\}\in [0,1]^{d}$ and computing $$\tilde{I}(f,\mathbf{P})=\frac{1}{N}\sum_{i=1}^{N}f(x_{i}).$$

Under mild conditions on the regularity of $f$, $\tilde{I}(f,\mathbf{P})$ tends to $I(f)$ with probability 1, i.e., for any $\epsilon>0$,

$$\lim_{N\rightarrow\infty}\mathop{\mathbb{P}}\limits_{x_1,x_2,\ldots,x_N}
\Big\{\Big|I(f)-\tilde{I}(f,\mathbf{P})\Big|>\epsilon\Big\}\rightarrow 0.$$

Monte Carlo (MC) method approximates integral $I(f)$ through average of randomly distributed sampling data collected on a discrete set, order of convergence $O(N^{-\frac{1}{2}})$ can be achieved, which also implies mean square error $O(N^{-1})$, see \cite{GSWZ2020,DKS2013}. 

We mainly adopt the discrepancy theory to estimate the approximation bounds. Firstly, we give the definitions of star discrepancy and $L_p-$discrepancy.

\textbf{Star discrepancy}. The \textbf{star discrepancy} of a sampling set $P_{N, d}=\{t_{i}\}_{1\leq i\leq N},t_i \\ \in [0,1]^{d}$ is defined by:

$$\label{*-d}D_{N}^{*}\left(t_{1}, t_{2}, \ldots, t_{N}\right):=\sup_{B\subset\mathscr{B}}|\frac{A(B;N;P_{N, d})}{N}-\lambda(B)|,$$
where $A(B;N;P_{N, d})$ denotes the number of points from $P_{N, d}$ that belongs to the rectangle $B$ anchored at $0$, $\mathscr{B}$ denotes the collection of all rectangles $B$ and $\lambda(B)$ denotes the Lebesgue measure of $B$.

The research of the star discrepancy can be divided into two aspects, one is to improve the star discrepancy bounds of suitable regime for $N$ and $d$(generally polynomial dependence, i.e., $N=O(d^{\alpha}), \alpha\ge 1$), which is also called pre-asymptotic bound, classical results of the star discrepancy upper bounds could reach the order of convergence $O(N^{-\frac{1}{2}})$, which also involves the use of random samples (uniformly distributed in $[0,1]^{d}$), see \cite{Ais2011,AH2014,HNWW2001}. The other is the deterministic design using Quasi-Monte Carlo (QMC) point sets such as Halton, Hammersley, Niederreiter point sets etc. \cite{Ht1960,Ham1960,Nie1988}, which is closely related to Quasi-Monte Carlo integration approximation according to the famous Koksma-Hlawka inequality, which is given by:

\begin{equation}\label{K-H}
\left|\int_{[0,1]^{d}} f(x) d x-\frac{1}{N} \sum_{t \in P_{N, d}} f(t)\right| \leq D_{N}^{*}\left(t_{1}, t_{2}, \ldots, t_{N}\right) V(f),\end{equation}
where $D_{N}^{*}\left(t_{1}, t_{2}, \ldots , t_{N}\right)$ is the star discrepancy of $P_{N, d}$ and $V(f)$ is the total variation of $f$ in the sense of Hardy and Krause. A smaller upper bound of star discrepancy means a smaller upper bound of approximation error in \eqref{K-H}. However the Hardy-Krause condition in the Koksma-Hlawka inequality seems to be rather strict. It works well for smooth functions, but it cannot be applied to most functions with simple discontinuities. For example, the characteristic function of a convex polyhedron has bounded Hardy-Krause variation only if the polyhedron is a $d-$dimensional interval, see \cite{BCGT2013}. Therefore, it is sometimes unrealistic to expect bounded variation which serves a good approximation of \eqref{K-H} in most function spaces, including the function space $\mathcal{H}^{\mathbf{1}}(K)$ and $F_{d,q}^{*}$ mentioned in this paper. In many cases, we also call Quasi-Monte Carlo point sets \textbf{low discrepancy point sets}. For a point set $\mathscr{P}$, the convergence order could reach $O((\ln N)^{\alpha_{d}}/N)$ for fixed dimension $d$ as $N\rightarrow \infty$, where $\alpha_{d}\ge 0$ are constants depending on dimension $d$. Examples of such point sets can be found in \cite{DP2010,Nie1992}. For applications of these point sets, see \cite{APC2016,CM2004,Lai1998,Lai2009}.

In recent years, random sampling has become a rather active area of research, due to its simplicity, flexibility and effectiveness, researchers investigate random sampling for different function spaces \cite{AST2004,AST2005,BG2004,BG2010,FX2019}. Besides, centered discrepancy of random sampling and Latin hypercube sampling are investigated in \cite{FMW2002}. Motivated by these developments, we incorporate a random viewpoint into our study of discrepancy theory and point distribution to consider random sampling under equal measure partition. For special case isometric grid partition, we assign each subcube only one sampling point, this method is called \textbf{stratified sampling} or \textbf{jittered sampling}, see \cite{Glass2004, PS2016}.

\textbf{$L_{p}-$discrepancy}. For a sampling set $P_{N, d}=\{t_{1}, t_{2}, \ldots , t_{N}\}$, $L_{p}-$discrepancy is defined by

$$L_{P}(D_{N},P_{N, d})=\Big(\int_{[0,1]^{d}}|z_{1}z_{2}\ldots z_{d}-
\frac{1}{N}\sum_{i=1}^{N}\mathbf{1}_{[0,z)}(t_{i})|^{p}dz\Big)^{1/p},$$
where $1\leq p<\infty$, $\mathbf{1}_{A}$ denotes the characteristic function on set $A$.
Classical applications of $L_{p}-$discrepancy are closely related to the worst case error of multivariate integration for the Sobolev class of functions that are once differentiable in each variable with finite $L_q-$norm, where $\frac{1}{p}+\frac{1}{q}=1$, see \cite{Nie1992}. If $p=\infty$, $L_{p}-$discrepancy will come back to the star discrepancy. Among $L_p-$discrepancy for different values $p\ge 1$, $p=2$ and $p=\infty$ are the most widely studied. The $L^p-$discrepancy bound of certain constructed point set has been intensively studied and many precise results are known. Lower bounds by Roth \cite{Roth1954} and Schmidt \cite{Schm1977} and upper bounds by Chen and Skriganov \cite{CS2002} and Skriganov \cite{Skri2006} via explicit constructions show the convergence order $$O\Big(\frac{(\ln N)^{\frac{d-1}{2}}}{N}\Big)$$ for $1<p<\infty$. For studies on $L_p-$discrepancy of random samples in \cite{ZD2016}, an upper bound on the $p-$moment of the $L_p-$discrepancy $(E[N^p\cdot L_{P}^{p}(D_{N},P_{N, d})])^{1/p}$ for $2\leq p\leq \infty$ is derived by the acceptance-rejection sampler using stratified inputs, which is of order $$O\Big(N^{(1-1/d)(1-1/p)}\Big)$$ and the constant of their bound depends on dimension $d$, index $p$, acceptance-rejection sampler set $A$ and Minkowski content related to $A$. A strong law of large numbers for integration on digital nets randomized by a nested uniform scramble is provided in \cite{OR2021}. For smooth enough function, they obtain asymptotically better convergence order $O(N^{-3+\epsilon}),\epsilon>0$ of mean square error than MC. For $f\in L^2([0,1]^{d})$, the asymptotically convergence order of mean square error is $O(N^{-1})$. Besides, the strong law of large numbers of randomized QMC is also proved by using the upper bound of $p-$moment of integral error in $L^p-$space. Jittered sampling construction gives rise to a set whose expected squared $L_2-$discrepancy is smaller than that of purely random points, see \cite{PS2016}. The similar result for $L_p-$discrepancy is obtained in \cite{KP2021}. A theoretical conclusion that the jittered sampling does not have the minimal expected $L_2-$discrepancy among all stratified samples from convex equivolume partitions with the same number of points is presented in \cite{KP1}, while the same conclusion for expected $L_p-$discrepancy is still an open problem.

Furthermore, there are some commonly used random sampling strategies. For example, simple random sampling, stratified sampling, Latin hypercube sampling, upper bound of star discrepancy for simple random sampling is studied in \cite{AH2014}, worst-case error bounds with high probability of least square approximation based on simple random samples is given in \cite{KUV2021}, variances of certain functions of stratified and Latin hypercube samples are studied in \cite{MCB1979,Stein1987}. Moreover, a method to reduce the clumping of the $X$-axis and $Y$-axis projections by imposing an additional $N$-rooks (Latin hypercube) constraint on the jittered sampling (stratified sampling) pattern is presented in \cite{CSW1994}.

Former research on uniform integration approximation using random sampling in general Sobolev space $\mathbb{H}^{s}$ with smoothness parameter $s>\frac{d}{2}$ defined over the unit sphere $\mathbb{S}^{d}$, see \cite{BSSW2014}, expected value of the squared uniform integration approximation for random samples collected on $d-$dimensional unit sphere $\mathbb{S}^{d}$ is given, relatively result using stratified samples based on equal measure partition to $\mathbb{S}^{d}$ is also presented. In this paper, the idea of stratified sampling by equal measure partition to $[0,1]^{d}$ is adopted to improve the classical mean square error of MC to $O(N^{-1-\frac{1}{d}})$ for functions in Sobolev space $\mathcal{H}^{\mathbf{1}}(K)$ equipped with a reproducing kernel. In more general Sobolev space $F_{d,q}^{*}$, see \cite{NW2010}, for functions equipped with some boundary conditions, we obtain upper bounds of $p$-moment of integral error.

The rest of this paper is organized as follows. In Section \ref{prelim} we first introduce some preliminaries, which are on reproducing kernel space and general equal measure partition. In Section \ref{app1} we present several improved mean square error bounds of uniform integration approximation in a Sobolev space $\mathcal{H}^{\mathbf{1}}(K)$ using stratified random sampling method according to different equal measure partition manners. In Section \ref{app2} we give several upper bounds of $p-$moment of integral approximation error in general Sobolev space $F_{d,q}^{*}$. Finally, in Section \ref{conclu} we conclude the paper with a short summary.

\section{Preliminaries on reproducing kernel Hilbert space and general equal measure partition}\label{prelim}

Before introducing the main result, we list preliminaries used in this paper. Firstly, reproducing kernel Hilbert space is introduced, we adopt the definitions in \cite{CZ2007}.

\begin{defn}
Let $X$ be a metric space, we say that $K:X\times X\rightarrow \mathbb{R}$ is symmetric when $$K(x,t)=K(t,x)$$ for all $x,t\in X$, and that it is positive semidefinite when for all finite sets $x=\{x_1,x_2,\ldots,x_N\}\subset X$, the $N\times N$ matrix $K[x]$ whose $(i,j)$ entry $K(x_i,x_j)$ is positive semidefinite. We say that $K$ is a Mercer kernel if it is continuous, symmetric and positive semidefinite. The matrix $K[x]$ above is called the Gramian of $K$ at $x$.
\end{defn}

For $x\in X,$ we denote by $K_x$ the function

\begin{align*}
K_x:&X \rightarrow \mathbb{R}\\& t\mapsto K(x,t).
\end{align*}

The main result of the reproducing kernel Hilbert space is the following:

\begin{thm}
There exits a unique Hilbert space $(\mathcal{H}(K),\langle,\rangle_{\mathcal{H}(K)})$ of functions on $X$ satisfying the following conditions:

(i) for all $x\in X,K_{x}\in \mathcal{H}(K)$;

(ii)the span of the set $\{K_{x}|x\in X\}$ in dense in $\mathcal{H}(K)$;

(iii)for all $f\in \mathcal{H}(K)$ and $x\in X, f(x)=\langle K_{x},f\rangle_{\mathcal{H}(K)}$.

Then Hilbert space $\mathcal{H}(K)$ is said to be a reproducing kernel Hilbert space(RKHS), property (iii) is referred to as the reproducing property.
\end{thm}

Secondly, the definition of the equal measure partition for $[0,1]^{d}$ is talked about in \cite{Beck1984} and discussed in \cite{PS2016} respectively, which is in the following.

For Lebesgue measure $\lambda$, there exists a partition $\Omega=\{\Omega_1,\Omega_2,\ldots,\Omega_N\}$ of $[0,1]^{d}$ into $N$ subsets $\Omega_j,1\leq j\leq N$ with the following properties:

$$
    [0,1]^{d}=\bigcup_{1\leq j\leq N}\Omega_{j}, \Omega_{j}\cap\Omega_{i}=\emptyset, j\neq i, \lambda(\Omega_{j})=\frac{1}{N}, 1\leq j\leq N,
$$

and

\begin{equation}\label{diam1}
    c_{1j}(d)N^{-\frac{1}{d}}\leq diam \Omega_{j} \leq c_{2j}(d)N^{-\frac{1}{d}}, 1\leq j\leq N,
\end{equation}
where for each subset $\Omega_{j}$, $c_{1j}(d)$ and $c_{2j}(d)$ are two constants depending only on dimension $d$, $diam A=\sup\{\theta(x,y), x,y \in A\}$ denotes the diameter of a set $A\subset [0,1]^{d}$, $\theta(\cdot,\cdot)$ is an Euclidean metric on $[0,1]^{d}$. 

Let

\begin{equation}\label{c1d}
    c_1(d)=\min \{c_{11}(d),c_{12}(d),\ldots,c_{1N}(d)\},
\end{equation}
and

\begin{equation}\label{c2d}
    c_2(d)=\max \{c_{21}(d),c_{22}(d),\ldots,c_{2N}(d)\}.
\end{equation}

Then from \eqref{c1d} and \eqref{c2d}, \eqref{diam1} means

\begin{equation}\label{diam2}
    c_{1}(d)N^{-\frac{1}{d}}\leq diam \Omega_{j} \leq c_{2}(d)N^{-\frac{1}{d}}, 1\leq j\leq N,
\end{equation}
where $c_{1}(d)$ and $c_{2}(d)$ are two constants depending only on dimension $d$.

We now consider a rectangle $R$ in $[0,1]^{d}$ anchored at $0$. For a partition $\Omega=\{\Omega_{1},\Omega_{2},
\ldots,\Omega_{N}\}$ of $[0,1]^{d}$, we put

\begin{equation}\label{jn}
I_{N}=\{j:\partial R \cap \Omega_{j}\neq \emptyset\},
\end{equation}
where $\partial R$ is the boundary of $R$.

Denote the cardinality of the index set $I_{N}$ by $|I_{N}|$ , we have the following estimation

\begin{equation}\label{CNbd0}
|I_{N}|\leq d\cdot c_2(d)\cdot N^{1-\frac{1}{d}}.
\end{equation}

In fact, 
let $R=[0,x)=\prod_{i=1}^{d}[0,x_i), R'=[0,y)=\prod_{i=1}^{d}[0,y_i)$ such that $x_i-y_i=c_2(d)\cdot N^{-\frac{1}{d}}, 1\leq i\leq d$, $R^{*}=\bigcup_{j\in I_{N}}\Omega_{j}$, thus $\lambda(R^{*})=\frac{|I_{N}|}{N}$. The union $R^{*}$ is a subset in this region $R\setminus R'.$ Therefore, we obtain

\begin{equation}\label{CNbd}
   |I_{N}|\leq N\lambda(R)-N\lambda(R').
\end{equation}

Besides, we have

\begin{equation}\label{CNbd2}
    \begin{aligned}
    &\lambda(R)-\lambda(R')\\=&\lambda([0,x))-\lambda([0,y))\\=&x_{1}x_{2}\ldots x_{d}-y_{1}y_{2}\ldots y_{d}\\=&\sum_{k=1}^{d}(y_{1}y_{2}\ldots y_{k-1}x_{k}x_{k+1}\ldots x_{d}-y_{1}y_{2}\ldots y_{k-1}y_{k}x_{k+1}\ldots x_{d})\\ \leq& d\cdot c_2(d)\cdot N^{-\frac{1}{d}}.
    \end{aligned}
\end{equation}
Combining with \eqref{CNbd} and \eqref{CNbd2}, we obtain \eqref{CNbd0}.

First case is isometric grid partition, we assign each subcube one random point, this forms jittered sampling.

\textbf{Case 1: Isometric grid partition.}

The simple case is the isometric grid partition, see \cite{Owenhpge} and \cite{PS2016}. The cube  $[0,1]^{d}$ is divided into $N$ axis parallel boxes $Q_{i},1\leq i\leq N,$ each with sides $\frac{1}{m}$. In this case, $diam Q_{i}=\frac{\sqrt{d}}{m},$ that is $c_1(d)=c_2(d)=\sqrt{d}$ in \eqref{diam2}.

For the isometric partition $\Omega=\{Q_{1},Q_{2},
\ldots,Q_{N}\}$ of $[0,1]^{d}$, let

$$
J_{N}=\{j:\partial R \cap Q_{j}\neq \emptyset\},
$$
where $\partial R$ denotes the partial of the rectangle $R$ in $[0,1]^{d}$ anchored at 0, following the step from \eqref{CNbd} to \eqref{CNbd2}, we have the estimation

$$
|J_{N}|\leq d\cdot N^{1-\frac{1}{d}}.
$$

Second case is non-isometric grid partition, the difference between it and isometric partition is that different isometric divisions are taken on each coordinate axis.

\textbf{Case 2: Non-isometric grid partition.}

In practice, for grid partition, the same value of $m$ on every dimension can be extended to use an $m_1\times m_2\times\ldots\times m_d$ grid of strata, see introduction in \cite{Owenhpge}. 

That is, if we choose only one sample point in each stratum, we have

$$
    N=\prod_{i=1}^{d}m_i.
$$

In this case, $[0,1]^{d}$ is divided into $N$ axis parallel boxes $Q'_{i},1\leq i\leq N,$ each with

$$
    \lambda(Q'_{i})=\frac{1}{N}=\frac{1}{\prod_{i=1}^{d}m_i},
$$
and

\begin{equation}\label{diamQ1}
    diam Q'_{i}=\sqrt{\sum_{i=1}^{d}\frac{1}{m_i^2}}.
\end{equation}

Thus, from \eqref{diamQ1} and fundamental inequality, we have

$$
   \frac{\sqrt{d}}{N^{\frac{1}{d}}}\leq diam Q'_{i}\leq \frac{c_0(d)}{N^{\frac{1}{d}}},
$$
where $c_0(d)$ satisfies

$$
    \sqrt{d}\leq c_0(d)\leq \sqrt{d}\cdot M,
$$
and $M$ is a constant such that

$$
    1\leq N\leq M^d.
$$

Third case is Hilbert space filling curve-based sampling, abbreviated as HSFC-based sampling, we mainly adopt the definition and notations in \cite{HO2016,HZ2019}.

\textbf{Case 3: Hilbert space filling curve-based sampling}

Hilbert space filling curve-based sampling (HSFC-based sampling) is actually a stratified sampling formed by a special partition manner. We will use the definition and properties in \cite{HO2016,HZ2019}. Let $y_i$ be the first $N=b^m$ points of the van der Corput sequence (van der Corput 1935) in base $b\ge 2$, $m=0,1,\ldots$. The integer $i-1\ge 0$ is written in base $b$ as $$i-1=\sum_{j=1}^{\infty}y_{ij}b^{j-1}$$ for $y_{ij}\in \{0,\ldots, b-1\}$. Then, $y_i$ is defined by $$y_i=\sum_{j=1}^{\infty}y_{ij}b^{-j}.$$ The scrambled version of $y_1,y_2,\ldots, y_N$ is $x_1,x_2,\ldots,x_N$ written as $$x_i=\sum_{j=1}^{\infty}x_{ij}b^{-j},$$ where $x_{ij}$ are defined through random permutations of the $y_{ij}$. These permutations depend on $y_{ik}$, for $k<j$. More precisely, $x_{i1}=\pi(y_{i1}),x_{i2}=\pi_{y_{i1}}(y_{i2})$ and generally for $j\ge 2$, $$x_{ij}=\pi_{y_{i1}\ldots y_{ij-1}}(y_{ij}).$$ Each random permutation is uniformly distributed over the $b!$ permutations of $\{0,\ldots,b-1\}$, and the permutations are mutually independent. The data values in the scrambled sequence
can be reordered such that $$x_i\sim U(I_{i}),$$ independently with $$I_i=[\frac{i-1}{N},\frac{i}{N}]$$ for $i=1,2,\ldots, N(=b^m)$. From \cite{HO2016}, we could use Hilbert mapping $H(x)$ from $[0,1]$ to $[0,1]^{d}$ for $d\ge 1,$ to map one dimensional uniformly distributed samples to $d-$dimensional uniformly distributed samples. We call this sampling manner Hilbert space filling curve-based sampling(HSFC-based sampling).
Let

\begin{equation}\label{ehi}
    E_i=H(I_i):=\{H(x)|x\in I_i\}.
\end{equation}
Then from Property 3 of HSFC in \cite{HO2016}, we obtain

\begin{equation}\label{xhuei}
    X_{i}=H(x_i)\sim U(E_i).
\end{equation}
This implies HSFC-based sampling is actually a stratified sampling, $\{E_i\}_{i=1}^{N}$ is a partition of $[0,1]^{d}$.

\section{Mean square error bounds of uniform integration approximation in Sobolev space $\mathcal{H}^{\mathbf{1}}(K)$}\label{app1}

In this section, we give the following results which are mean square error bounds of uniform integration approximation for functions in Sobolev space $\mathcal{H}^{\mathbf{1}}(K)$ equipped with a reproducing kernel, we adopt the definition of $\mathcal{H}^{\mathbf{1}}(K)$ in \cite{DP2014}. 

Let $$\mathcal{H}^{\mathbf{1}}:=\mathcal{H}^{(1,1,\ldots,1)}([0,1]^{d})$$ be the Sobolev spaces on $[0,1]^{d}$. $\forall f\in \mathcal{H}^{\mathbf{1}}$, we have $$\frac{\partial^{d}}{\partial x}f(x)\in \mathcal{H}([0,1]^{d}),$$ where $\partial x=\partial x_1\partial x_2\ldots,\partial x_d$, $\mathcal{H}([0,1]^{d})$ denotes the Hilbert space. 
Then for $f,g \in \mathcal{H}^{\mathbf{1}}$, we define the following inner product for the Hilbert space $\mathcal{H}([0,1]^{d})$,

\begin{equation}\label{rkhsinpro}
\langle f,g\rangle_{\mathcal{H}^{\mathbf{1}}}
=\int_{[0,1]^{d}}\frac{\partial^{d}f}{\partial x}(t)\frac{\partial^{d}g}{\partial x}(t)dt.
\end{equation}

Thus, we set $\|f\|_{\mathcal{H}^{\mathbf{1}}}=\langle f,f\rangle_{\mathcal{H}^{\mathbf{1}}}^{1/2}$ be the norm induced by the inner product defined in \eqref{rkhsinpro}.
We now define a reproducing kernel in $\mathcal{H}^{\mathbf{1}}$, which is given by

\begin{equation}\label{specialkernel1}
K(x,y)=\int_{[0,1]^{d}}\mathbf{1}_{(x,1]}(t)
\mathbf{1}_{(y,1]}(t)dt,
\end{equation}
where $x=(x_1,x_2,\ldots,x_d), y=(y_1,y_2,\ldots,y_d)$, $(x,1]=\prod_{i=1}^{d}(x_i,1],(y,1]=\prod_{i=1}^{d}(y_i,1]$, and $\mathbf{1}_{A}$ denotes the characteristic function on set $A$. $\mathcal{H}^{\mathbf{1}}(K)$ denotes the Sobolev space $\mathcal{H}^{\mathbf{1}}$ equipped with a reproducing kernel function $K(x,y)$ defined in \eqref{specialkernel1}.  Correspondingly, in \eqref{rkhsinpro}, we define $\langle f,g\rangle_{\mathcal{H}^{\mathbf{1}}}=\langle f,g\rangle_{\mathcal{H}^{\mathbf{1}}(K)}$.

Easy to check that for kernel function defined in \eqref{specialkernel1}, the reproducing property is satisfied, that is,

$$
\langle f,K(\cdot,y)\rangle_{\mathcal{H}^{\mathbf{1}}(K)}=
\int_{[0,1]^{d}}\frac{\partial^{d}f}{\partial x}(t)\frac{\partial^{d}K(x,y)}{\partial x}(t)dt=f(y).
$$

\begin{thm}\label{epecintap}
Given a partition $\Omega=\{\Omega_{1},\Omega_{2},
\ldots,\Omega_{N}\}$ of the unit cube $[0,1]^{d}$, any $d \ge 2$ and $N\in \mathbb{N}$, samples $Y_{1}, Y_{2}, Y_{3}, \ldots, Y_{N}$ are uniformly distributed in the subset $\Omega_{1}, \Omega_{2}, \Omega_{3}, \ldots, \Omega_{N}$ which forms an equal measure partition of $[0,1]^{d}$, then we have

\begin{equation}
\mathbb{E}[\sup_{f\in \mathcal{H}^{\mathbf{1}}(K),\|f\|_{\mathcal{H}^{\mathbf{1}}(K)}\leq 1 }\Big|\frac{1}{N}\sum_{n=1}^{N}f(Y_n)-\int_{[0,1]^{d}}f(x)dx\Big|^2]
\leq \frac{d\cdot c_{2}(d)}{N^{1+\frac{1}{d}}},
\end{equation}
where $c_{2}(d)$ is defined in \eqref{c2d} which is related to the maximum diameter of $\Omega_j,1\leq j\leq N$, $f$ is a function in Sobolev space $\mathcal{H}^{\mathbf{1}}(K)$.
\end{thm}

\begin{proof}

We consider the relationship between multivariate integration approximation and $L_2-$discrepancy. That is, for $f\in \mathcal{H}^{\mathbf{1}}(K)$, we have

\begin{equation}\label{intappholder1}
\begin{aligned}
&\Big|\frac{1}{N}\sum_{n=1}^{N}f(x_n)-\int_{[0,1]^{d}}f(x)dx\Big|\\=&
\Big|\int_{[0,1]^{d}}\langle f,K(\cdot,x)\rangle_{\mathcal{H}^{\mathbf{1}}(K)}dx-\frac{1}{N}\sum_{n=1}^{N}
\langle f,K(\cdot,x_n)\rangle_{\mathcal{H}^{\mathbf{1}}(K)}\Big|\\=&\Big|\langle f, \int_{[0,1]^{d}}K(\cdot,x)dx-\frac{1}{N}
\sum_{n=1}^{N}K(\cdot,x_n)\rangle_{\mathcal{H}^{\mathbf{1}}(K)}\Big|\\ \leq& \|f\|_{\mathcal{H}^{\mathbf{1}}(K)}\|h\|_{\mathcal{H}^{\mathbf{1}}(K)},
\end{aligned}
\end{equation}
where

\begin{equation}\label{hz1}
    h(z)=\int_{[0,1]^{d}}K(z,x)dx-\frac{1}{N}\sum_{n=1}^{N}K(z,x_n).
\end{equation}

Putting \eqref{specialkernel1} into \eqref{hz1}, we have

\begin{align*}
h(z)&=\int_{[0,1]^{d}}\int_{[0,1]^{d}}
\mathbf{1}_{(z,1]}(t)\mathbf{1}_{(x,1]}(t)dtdx-
\frac{1}{N}\sum_{n=1}^{N}\int_{[0,1]^{d}}
\mathbf{1}_{(z,1]}(t)\mathbf{1}_{(x_n,1]}(t)dt\\&=
\int_{[0,1]^{d}}\mathbf{1}_{(z,1]}(t)\int_{[0,1]^{d}}
\mathbf{1}_{(x,1]}(t)dxdt-\int_{[0,1]^{d}}\frac{1}{N}
\sum_{n=1}^{N}\mathbf{1}_{(x_n,1]}(t)\mathbf{1}_{(z,1]}(t)dt
\\&=\int_{[0,1]^{d}}\mathbf{1}_{(z,1]}(t)\Big(\int_{[0,1]^{d}}
\mathbf{1}_{(x,1]}(t)dx-\frac{1}{N}
\sum_{n=1}^{N}\mathbf{1}_{(x_n,1]}(t)\Big)dt\\&=
-\int_{[0,1]^{d}}\mathbf{1}_{(z,1]}(t)\Big(\frac{1}{N}
\sum_{n=1}^{N}\mathbf{1}_{[0,t)}(x_n)-\lambda([0,t))\Big)dt.
\end{align*}

Thus

\begin{equation}\label{hhhk}
\langle h,h \rangle_{\mathcal{H}^{\mathbf{1}}(K)}=\int_{[0,1]^d}\frac{\partial^{d}h}{\partial z}(t)\frac{\partial^{d}h}{\partial z}(t)dt=L_2^2(D_N,x).
\end{equation}

Combining \eqref{intappholder1} and \eqref{hhhk} , we have

\begin{equation}\label{rkhsappro1}
\Big|\frac{1}{N}\sum_{n=1}^{N}f(x_n)-\int_{[0,1]^{d}}f(x)dx\Big|
\leq L_{2}(D_N,x)\|f\|_{\mathcal{H}^{\mathbf{1}}(K)},
\end{equation}
then the estimation comes down to $L_{2}(D_N,x)$.

We consider the following discrepancy function,

\begin{equation}\label{dispfunc1}
\Delta_{\mathscr{P}}(x)=\frac{1}{N}
\sum_{n=1}^{N}1_{[0,x)}(Y_n)-\lambda([0,x)),
\end{equation}
where $Y_{n},1\leq n\leq N$ denotes the uniformly distributed samples from each subset $\Omega_{n},1\leq n\leq N$ of general partition of $[0,1]^{d}$, and $\lambda([0,x))$ denotes the Lebesgue measure of the anchored axis-parallel box $[0,x)$.
For an anchored axis-parallel box $R_0=[0,x)$, we can break it into two parts, one is the disjoint union of $\Omega_{i}$ entirely contained by $R_0$ and the union of remaining pieces which are the intersections of some $\Omega_{j}$ and $R_0$, which is,

$$
R_0=\bigcup_{i\in I_0}\Omega_{i}\cup\bigcup_{j\in J_0}(\Omega_{j}\cap R_0),
$$
where $I_0,J_0$ are two index-sets.

Samples $Y_{1}, Y_{2}, Y_{3}, \ldots, Y_{N}$ are uniformly distributed in the subset $\Omega_{1}, \Omega_{2}, \Omega_{3}, \ldots, \\\Omega_{N}$ of $[0,1]^{d}$, from \eqref{dispfunc1}, easy to know that the discrepancy function equals 0 for the disjoint union of $\Omega_{i}$ entirely contained by $R_0$, thus we only need to consider the union of remaining pieces, we set it $T_0$, and from \eqref{CNbd0}, we have $\lambda(T_0)\leq d\cdot c_2(d)\cdot N^{-\frac{1}{d}}$. Thus from \eqref{dispfunc1}, we have

\begin{align*}
\Delta_{\mathscr{P}}(x)&=\frac{1}{N}
\sum_{n=1}^{N}\mathbf{1}_{[0,x)}(Y_n)-\lambda([0,x))\\&=\frac{1}{N}
\sum_{n=1}^{N}\mathbf{1}_{T_0}(Y_n)-\lambda(T_0).
\end{align*}

Therefore,

\begin{equation}\label{exl2dpb1_1}
\mathbb{E}(L_2^2(\mathscr{P}))=\mathbb{E}
(\int_{[0,1]^d}|\frac{1}{N}
\sum_{n=1}^{N}\mathbf{1}_{T_0}(Y_n)-\lambda(T_0)|^2dx).
\end{equation}

Consider the whole summation as a random variable which defines on a region we let it $P_{\Omega}$, besides we set the probability measure be $w$, thus from \eqref{exl2dpb1_1}, we have

\begin{align*}
\mathbb{E}(L_2^2(\mathscr{P}))&=\int_{P_{\Omega}}
\int_{[0,1]^d}|\frac{1}{N}
\sum_{n=1}^{N}\mathbf{1}_{T_0}(Y_n)-\lambda(T_0)|^2dxdw\\&
=\int_{[0,1]^d}\int_{P_{\Omega}}|\frac{1}{N}
\sum_{n=1}^{N}\mathbf{1}_{T_0}(Y_n)-\lambda(T_0)|^2dwdx.
\end{align*}

Furthermore,

$$
\mathbb{E}(\frac{1}{N}
\sum_{n=1}^{N}\mathbf{1}_{T_0}(Y_n))=\int_{P_{\Omega}}\frac{1}{N}
\sum_{n=1}^{N}\mathbf{1}_{T_0}(Y_n)dw=\lambda(T_0).
$$

Thus,

\begin{equation}\label{expvar1}
\mathbb{E}(L_2^2(\mathscr{P}))=Var(\frac{1}{N}
\sum_{n=1}^{N}\mathbf{1}_{T_0}(Y_n)),
\end{equation}
where $Var(\frac{1}{N}
\sum_{n=1}^{N}\mathbf{1}_{T_0}(Y_n))$ denotes the variance. Besides,

\begin{align*}
Var(\sum_{n=1}^{N}\mathbf{1}_{T_0}(Y_n)):&=
\sum_{i=1}^{N}[\mathbf{E}\left(\mathbf{1}_{T_0}^2\left(Y_{i}\right)\right)-\left(\mathbf{E}\left(\mathbf{1}_{T_0}\left(Y_{i}\right)\right)\right)^{2}]\\&=\sum_{i=1}^{N}[\mathbb P\left(Y_{i}\in T_0\cap \Omega_i\right)-\mathbb P^2\left(Y_{i}\in T_0\cap \Omega_i\right)]\\&=\sum_{i=1}^{N}\frac{\lambda(T_0\cap\Omega_i)}{\lambda(\Omega_i)}
(1-\frac{\lambda(T_0\cap\Omega_i)}{\lambda(\Omega_i)})
\\&\leq N\lambda\left(T_0\right)-N\lambda^{2}(T_0).
\end{align*}

Combining with $\lambda(T_0)\leq d\cdot c_2(d)\cdot N^{-\frac{1}{d}}$ and \eqref{expvar1}, we have

\begin{equation}\label{exl2}
\mathbb{E}(L_2^2(\mathscr{P}))=Var(\frac{1}{N}
\sum_{n=1}^{N}\mathbf{1}_{T_0}(Y_n))\leq \frac{d\cdot c_2(d)}{N^{1+\frac{1}{d}}}.
\end{equation}

From \eqref{rkhsappro1}, we have

\begin{equation}\label{lastfomu1}
\mathbb{E}[\sup_{f\in \mathcal{H}^{\mathbf{1}}(K),\|f\|_{\mathcal{H}^{\mathbf{1}}(K)}\leq 1 }|\frac{1}{N}\sum_{n=1}^{N}f(Y_n)-\int_{[0,1]^{d}}f(x)dx|^2]
\leq \mathbb{E}(L_2^2(\mathscr{P})),
\end{equation}

which complete the proof.

\end{proof}

\begin{rem} 
Theorem \ref{epecintap} adopts technique of stratified random sampling formed by general equal measure partition to obtain better convergence order $O(N^{-1-\frac{1}{d}})$ of mean square error bounds, comparing with traditional convergence order $O(N^{-1})$ using crude Monte Carlo method.
\end{rem}
Combining with Theorem \ref{epecintap} and Example 1 in Section 1, we obtain the following corollary.

\begin{cor}\label{edgp1}
For an isometric grid partition $\{Q_{1},Q_{2},
\ldots,Q_{N}\}$ of the unit cube $[0,1]^{d}$, any $d,m \ge 2$ and $N\in \mathbb{N}$ such that $N=m^{d}$, $d-$dimension samples $X_{1}, X_{2}, X_{3},\\ \ldots, X_{N}$ are uniformly distributed in the subset $Q_{1}, Q_{2}, Q_{3}, \ldots, Q_{N}$ of $[0,1]^{d}$, then we have

\begin{equation}
\mathbb{E}[\sup_{f\in \mathcal{H}^{\mathbf{1}}(K),\|f\|_{\mathcal{H}^{\mathbf{1}}(K)}\leq 1 }\Big|\frac{1}{N}\sum_{n=1}^{N}f(X_n)-\int_{[0,1]^{d}}f(x)dx\Big|^2]
\leq \frac{d}{N^{1+\frac{1}{d}}},
\end{equation}
where $f$ is a function in Sobolev space $\mathcal{H}^{\mathbf{1}}(K)$.
\end{cor}

\begin{cor}\label{edgp0}
For the HSFC-based sampling, which is a special partition $\{E_{1},E_{2}\\,\ldots,E_{N}\}$ of the unit cube $[0,1]^{d}$, any $d,m,b \ge 2$ and $N\in \mathbb{N}$ such that $N=b^{m}$, $d-$dimension samples $X'_{1}, X'_{2}, X'_{3}, \ldots, X'_{N}$ are uniformly distributed in the subset $E_{1}, E_{2}, E_{3}, \ldots, E_{N}$ of $[0,1]^{d}$, then we have

\begin{equation}
\mathbb{E}[\sup_{f\in \mathcal{H}^{\mathbf{1}}(K),\|f\|_{\mathcal{H}^{\mathbf{1}}(K)}\leq 1 }\Big|\frac{1}{N}\sum_{n=1}^{N}f(X'_n)-\int_{[0,1]^{d}}f(x)dx\Big|^2]
\leq \frac{2d\sqrt{d+3}}{N^{1+\frac{1}{d}}},
\end{equation}
where $f$ is a function in Sobolev space $\mathcal{H}^{\mathbf{1}}(K)$.
\end{cor}

\begin{proof}
From \cite{HZ2019}, if we let $A=H([p,q])$ for $0\leq p<q\leq 1$, then $\lambda_{d}(A)=\lambda_{1}([p,q])=q-p,$ where we use $\lambda_{d}$ and $\lambda_{1}$ to distinguish $d-$dimensional and $1-$dimensional Lebesgue measure. Furthermore, if $x\sim U([p,q])$, then $H(x)\sim U(A).$ Let $r$ be the diameter of $A$, then

\begin{equation}\label{uppei}
    r\leq 2\sqrt{d+3}\cdot (q-p)^{\frac{1}{d}}.
\end{equation}

From \eqref{ehi} and \eqref{xhuei}, we have, $\{E_1,E_2,\ldots, E_N\}$ is a partition of $[0,1]^{d}$, and $\lambda(E_i)=\frac{1}{N}, 1\leq i\leq N$, from \eqref{uppei}, the diameter of every $E_i$ is no large than $2\sqrt{d+3}\cdot N^{-\frac{1}{d}}$. This implies $c_2(d)= 2\sqrt{d+3}$ in \eqref{diam2}. Combining with Theorem \ref{epecintap}, the proof is completed.
\end{proof}

Next, we give an uniform mean square error bound for Latin hypercube samples. We use the definition of Latin hypercube sampling in \cite{Owenhpge}. Suppose

\begin{equation}\label{LHSP}
    X_{ij}=\frac{\pi_{j}(i-1)+U_{ij}}{N}, 1\leq i\leq N, 1\leq j\leq d,
\end{equation}
where $\pi_1,\pi_2,\ldots,\pi_d$ are uniform permutations of $\{0,1,\ldots,N-1\}$, $U_{ij}\sim\textbf{U}[0,1)$, and all the $U_{ij}$ and $\pi_j$ are independent, then $X_{ij}$ consist of Latin hypercube samples.

\begin{lma}\label{LHSIUD}\cite{Owenhpge}
Let $d\ge 1, N\ge 1$ be integers and $X_{ij}$ be a Latin hypercube sample defined by \eqref{LHSP}, then $X_i\sim\textbf{U}[0,1)^{d}$ holds for each $i=1,2,\ldots,N$.
\end{lma}

\begin{thm}\label{LHSC1}
Let $f(x)$ be a real-valued function in $\mathcal{H}^{\mathbf{1}}(K)$ with the condition $$C=\int_{[0,1]^{d}}(f(x)-\int_{[0,1]^d}f(x)dx)^2dx<\infty,$$ for isometric grid partition, $Z_{1},Z_{2},\ldots,Z_{N}$ are Latin hypercube samples, then we have

\begin{equation}
\mathbb{E}[\sup_{f\in \mathcal{H}^{\mathbf{1}}(K),\|f\|_{\mathcal{H}^{\mathbf{1}}(K)}\leq 1 }\Big|\frac{1}{N}\sum_{n=1}^{N}f(Z_n)-\int_{[0,1]^{d}}f(z)dz\Big|^2]
\leq \frac{C}{N-1},
\end{equation}
where $f$ is a function in Sobolev space $\mathcal{H}^{\mathbf{1}}(K)$.
\end{thm}

\begin{proof}
Considering, 

\begin{equation}\label{esupLh}
\mathbb{E}[\sup_{f\in \mathcal{H}^{\mathbf{1}}(K),\|f\|_{\mathcal{H}^{\mathbf{1}}(K)}\leq 1 }\Big|\frac{1}{N}\sum_{n=1}^{N}f(Z_n)-\int_{[0,1]^{d}}f(z)dz\Big|^2].
\end{equation}

Suppose the whole part in expectation formula \eqref{esupLh} as a random variable which defines on a region we let it $Z$, besides we set the probability measure be $\mu$, thus,

\begin{equation}\label{Esupsup1}
    \begin{aligned}
        &\mathbb{E}[\sup_{f\in \mathcal{H}^{\mathbf{1}}(K),\|f\|_{\mathcal{H}^{\mathbf{1}}(K)}\leq 1 }|\frac{1}{N}\sum_{n=1}^{N}f(Z_n)-\int_{[0,1]^{d}}f(z)dz|^2]\\&=\int_{Z}\sup_{f\in \mathcal{H}^{\mathbf{1}}(K),\|f\|_{\mathcal{H}^{\mathbf{1}}(K)}\leq 1 }|\frac{1}{N}\sum_{n=1}^{N}f(Z_n)-\int_{[0,1]^{d}}f(z)dz|^2d\mu\\&=\sup_{f\in \mathcal{H}^{\mathbf{1}}(K),\|f\|_{\mathcal{H}^{\mathbf{1}}(K)}\leq 1 }\int_{Z}|\frac{1}{N}\sum_{n=1}^{N}f(Z_n)-\int_{[0,1]^{d}}f(z)dz|^2d\mu.
    \end{aligned}
\end{equation}

According to Lemma \ref{LHSIUD}, we have

\begin{equation}\label{Esmin1}
    \mathbb{E}(\frac{1}{N}\sum_{n=1}^{N}f(Z_n))=\int_{[0,1]^{d}}f(z)dz.
\end{equation}

Combining \eqref{Esupsup1} and \eqref{Esmin1}, we have

\begin{equation}
  \begin{aligned}
    &\mathbb{E}[\sup_{f\in \mathcal{H}^{\mathbf{1}}(K),\|f\|_{\mathcal{H}^{\mathbf{1}}(K)}\leq 1 }|\frac{1}{N}\sum_{n=1}^{N}f(Z_n)-\int_{[0,1]^{d}}f(z)dz|^2]\\&=\sup_{f\in \mathcal{H}^{\mathbf{1}}(K),\|f\|_{\mathcal{H}^{\mathbf{1}}(K)}\leq 1 }Var(\frac{1}{N}\sum_{n=1}^{N}f(Z_n)).
    \end{aligned}
\end{equation}

Furthermore, from the condition $$C=\int_{[0,1]^{d}}(f(x)-\int_{[0,1]^d}f(x)dx)^2dx<\infty$$ and a result in \cite{Owen1997}, which is,

\begin{equation}
    Var(\frac{1}{N}\sum_{n=1}^{N}f(Z_n))\leq \frac{C}{N-1},
\end{equation}
which completes the proof.
\end{proof}

\begin{rem}
Systematic studies of Latin hypercube sampling on asymptotic variance for the sample mean are given in \cite{MCB1979,Owen1992,Owen1994,Owen1997}, in Theorem \ref{LHSC1}, we use result of asymptotic variance for the sample mean derived by \cite{Owen1997} to give the uniform integration approximation bound for Latin hypercube samples, we find that the convergence order is consistent with MC method, which is $O(N^{-1})$, stratified sampling improves the convergence order of these two sampling methods to $O(N^{-1-\frac{1}{d}})$.
\end{rem}
For simple random sampling mode, we have the following result, which is a case of stratified sampling degradation where $\Omega_1=\Omega_2=\ldots=\Omega_N=[0,1]^{d}$.

\begin{cor}\label{ransammod}
Let $x_{1},x_{2},\ldots,x_{N}$ be simple random sampling points uniformly distributed in $[0,1]^{d}$, then we have

\begin{equation}
\mathbb{E}[\sup_{f\in \mathcal{H}^{\mathbf{1}}(K),\|f\|_{\mathcal{H}^{\mathbf{1}}(K)}\leq 1 }\Big|\frac{1}{N}\sum_{n=1}^{N}f(x_n)-\int_{[0,1]^{d}}f(z)dz\Big|^2]\leq \frac{d^{\frac{3}{2}}}{N},
\end{equation}
where $f$ is a function in Sobolev space $\mathcal{H}^{\mathbf{1}}(K)$.
\end{cor}

\begin{rem}
Theorem \ref{epecintap}, Corollary \ref{edgp1},\ref{edgp0},\ref{ransammod} and Theorem \ref{LHSC1} imply the stratified random sampling formed by equal measure partition could obtain better mean square error bounds of uniform integration approximation than simple random sampling and Latin hypercube sampling in sense of convergence order. For instance, for stratification grid and HSFC-based sampling, which serve as special partition manners, we can improve the convergence order of mean square error bounds from traditional $O(N^{-1})$ using simple random sampling or Latin hypercube sampling to $O(N^{-1-\frac{1}{d}})$.
\end{rem}

\section{Upper bounds of $p$-moment of integral approximation error in general Sobolev space $F_{d,q}^{*}$}\label{app2}

In this section, we use expected $L_p-$discrepancy($p\ge 1$) bounds for stratified random samples formed by general equal measure partition to give several upper bounds of $p-$moment of integral approximation error for functions in Sobolev space $F_{d,q}^{*}$, where $\frac{1}{p}+\frac{1}{q}=1$, we adopt the definition of $F_{d,q}^{*}$ in \cite{NW2010}.

Let 

$$
W_{q}^{\mathbf{1}}:=W_{q}^{(1,1,\ldots,1)}([0,1]^{d})
$$
be the Sobolev spaces on $[0,1]^{d}$. For $f\in W_{q}^{\mathbf{1}}$, we define the norm

$$
\|f\|_{d,q}^{*}=\Big(\int_{[0,1]^{d}}|\frac{\partial^{d}}{\partial x}f(x)|^{q}dx\Big)^{1/q}
$$
for $q\in[1,\infty)$ and

$$
\|f\|_{d,\infty}^{*}=\sup_{x\in[0,1]^{d}}|\frac{\partial^{d}}{\partial x}f(x)|,
$$
where $\partial x=\partial x_1\partial x_2\ldots\partial x_d$.
We consider a following space, 

\begin{equation}\label{Sobospadef}
    F_{d,q}^{*}=\{f\in W_{q}^{\mathbf{1}}| f(x)=0 \ \text{if}\ x_j=1 \ \text{for some}\ 1\leq j\leq d, \|f\|_{d,q}^{*}<\infty\},
\end{equation}
which enforces the functions in $W_{q}^{\mathbf{1}}$ boundary conditions. Boundary conditions are necessary there in \eqref{Sobospadef} for uniform integration approximation, see \cite{NW2010}.

\begin{thm}\label{exunibd1}
For any $d \ge 2, 1\leq p<\infty,$ and $N\in \mathbb{N}$, $d-$dimensional sampling set $x=\{x_{1}, x_{2}, x_{3}, \ldots, x_{N}\}$ is uniformly distributed in the subset $\Omega_{1}, \Omega_{2}, \Omega_{3}, \ldots, \Omega_{N}$ which is some general partition of $[0,1]^{d}$, then for functions $f$ in Sobolev space $ F_{d,q}^{*}$ , where  $\frac{1}{p}+\frac{1}{q}=1$, we have

\begin{equation}
\mathbb{E}\big[\sup_{f\in F_{d,q}^{*},\|f\|_{d,q}^{*}\leq 1 }\Big|\frac{1}{N}\sum_{n=1}^{N}f(x_n)-\int_{[0,1]^{d}}f(x)dx\Big|^p\big]\leq \frac{d^{\frac{p}{2}}\cdot c_2(d)^{\frac{p}{2}}}{N^{\frac{p}{2}+\frac{p}{2d}}},
\end{equation}
where $c_{2}(d)$ is defined in \eqref{c2d} which is related to the maximum diameter of $\Omega_j,1\leq j\leq N$.
\end{thm}

\begin{proof}

For $f\in F_{d,q}^{*}$, due to the boundary conditions in \eqref{Sobospadef}, then using integration by parts, we have

$$
|I(f)-\tilde{I}(f,\mathbf{P})|=\Big|\int_{[0,1]^{d}}\Big(\prod_{k=1}^{d}z_{i}-
\frac{1}{N}\sum_{i=1}^{N}\mathbf{1}_{[0,z)}(x_{i})\Big)\frac{\partial^{d}}{\partial x}f(z)dz\Big|.
$$

Applying the H{\"o}lder inequality, we obtain the following uniform integration approximation in Sobolev space,

\begin{equation}\label{soboappx1}
\sup_{f\in F_{d,q}^{*},\|f\|_{d,q}^{*}\leq 1 }|I(f)-\tilde{I}(f,\mathbf{P})|=\Big(\int_{[0,1]^{d}}|z_{1}z_{2}\ldots z_{d}-
\frac{1}{N}\sum_{i=1}^{N}\mathbf{1}_{[0,z)}(x_{i})|^{p}dz\Big)^{1/p}.
\end{equation}

Then it suffices to estimate $L_{p}-$discrepancy for random samples under equal measure partition.

For an equal measure partition $\Omega=\{\Omega_{1}, \Omega_{2}, \Omega_{3}, \ldots, \Omega_{N}\}$ of $[0,1]^{d}$, point set $x=\{x_{1}, x_{2}, x_{3}, \ldots, x_{N}\}$ is uniformly distributed in the subset $\Omega_{1}, \Omega_{2}, \Omega_{3}, \ldots, \Omega_{N}$, then for a measurable subset $A$ of $\Omega_{i}$,

$$
\mathbb{P}(x_i\in A)=\frac{\lambda(A)}{\lambda(\Omega_{i})}=N\lambda(A).
$$

We now consider an axis parallel rectangle $R=[0,z)$ anchored at $0$ of $[0,1]^{d}$, let $\mathscr{I}$ denote the set of all values of $i$ for which the subsets $\Omega_{i}$ intersect the boundary $\partial R$ of $R$. For each $i\in \mathscr{I}$, we define the following Bernoulli distribution,

$$
    \eta_{i}=\left\{
\begin{aligned}
&1, x_{i}\in R\\&
0, otherwise.
\end{aligned}
\right.
$$

If we let $\xi_{i}=\eta_{i}-\mathbb{E}\eta_{i},1\leq i\leq N$, then we have

$$
\mathbb{E}\xi_{i}=0, |\xi_{i}|\leq 1, \forall 1\leq i\leq N.
$$

Furthermore, for equal measure partition, we have the following basic fact for discrepancy, that is, for axis parallel rectangle $R=[0,z)\in [0,1]^{d}$ anchored at zero, we can break it into two special parts, one is the disjoint union of $\Omega_{k}$ entirely contained by $R$ and the union of remaining pieces which are the intersections of some $\Omega_{i}$ and $\partial R$, that is,

$$
R=\bigcup_{k\in \mathscr{K}}\Omega_{k}\cup\bigcup_{i\in \mathscr{I}}(\Omega_{i}\cap R),
$$
where $\mathscr{K},\mathscr{I}$ are two index-sets.
Then we have

$$
\frac{1}{N}\sum_{i=1}^{N}\mathbf{1}_{[0,z)}(x_{i})-z_{1}z_{2}\ldots z_{d}=\frac{1}{N}
\sum_{i\in\mathscr{I}}\xi_{i}.
$$

Thus we have

$$
|z_{1}z_{2}\ldots z_{d}-
\frac{1}{N}\sum_{i=1}^{N}\mathbf{1}_{[0,z)}(x_{i})|^p=\frac{1}{N^p}
\sum_{i_{1}\in\mathscr{I}}\ldots
\sum_{i_{p}\in\mathscr{I}}\xi_{i_{1}}\ldots\xi_{i_{p}},
$$
and

\begin{equation}\label{lpdisfun3}
\mathbb{E}\Big(|z_{1}z_{2}\ldots z_{d}-
\frac{1}{N}\sum_{i=1}^{N}\mathbf{1}_{[0,z)}(x_{i})|^p\Big)=\frac{1}{N^p}
\sum_{i_{1}\in\mathscr{I}}\ldots
\sum_{i_{p}\in\mathscr{I}}\mathbb{E}
\Big(\xi_{i_{1}}\ldots\xi_{i_{p}}\Big).
\end{equation}

Let $s={|\mathscr{I}|\choose p}\leq |\mathscr{I}|^{p}$, for every selected $p$ indices $i_{1},i_{2},\ldots,i_{p}$ in each index set $\mathscr{I}$. Set

$$
\omega_{i}=\xi_{i_{1}}\ldots\xi_{i_{p}}
$$
for all $1\leq i\leq s.$
Due to $|\xi_{i}|\leq 1, \forall 1\leq i\leq N$, thus we have $|\omega_{i}|\leq 1, 1\leq i\leq s.$

The random variables $\xi_{i},i\in\mathscr{I}$ are independent as we have supposed, then if one of $i_{1},\ldots, i_{p}$ is different from others in \eqref{lpdisfun3}, we have

$$
\mathbb{E}
\Big(\xi_{i_{1}}\ldots\xi_{i_{p}}\Big)=
\mathbb{E}(\xi_{i_{j}})\mathbb{E}
\Big(\xi_{i_{1}}\ldots\xi_{i_{j-1}}\xi_{i_{j+1}}\ldots
\xi_{i_{p}}\Big)=0.
$$

It follows that only non-zero contribution to the sum \eqref{lpdisfun3} comes from those terms where each of $i_{1},\ldots, i_{p}$ appears more than once. Then the major contribution comes when they appear in pairs, and
there are at least $\sqrt{s}$ such pairs.
Such terms $\mathbb{E}(\omega_i),1\leq i\leq s$ are bounded by $1$, thus,

\begin{align*}
\mathbb{E}\Big(|z_{1}z_{2}\ldots z_{d}-
\frac{1}{N}\sum_{i=1}^{N}\mathbf{1}_{[0,z)}(x_{i})|^p\Big)&=\frac{1}{N^p}
\sum_{i_{1}\in\mathscr{I}}\ldots
\sum_{i_{p}\in\mathscr{I}}\mathbb{E}
\Big(\xi_{i_{1}}\ldots\xi_{i_{p}}\Big)\\&\leq \frac{|\mathscr{I}|^{p/2}}{N^{p}}.
\end{align*}

Consider the term $|z_{1}z_{2}\ldots z_{d}-
\frac{1}{N}\sum_{i=1}^{N}\mathbf{1}_{[0,z)}(x_{i})|^p$ as a random variable which defines on a region we let it $P_{\Omega}$, besides we set the probability measure be $w$, from Fubini's theorem, we have,

\begin{align*}
&\int_{P_{\Omega}}\int_{[0,1]^d}|z_{1}z_{2}\ldots z_{d}-
\frac{1}{N}\sum_{i=1}^{N}\mathbf{1}_{[0,z)}(x_{i})|^p dz dw\\=&\int_{[0,1]^d}\int_{P_{\Omega}}|z_{1}z_{2}\ldots z_{d}-
\frac{1}{N}\sum_{i=1}^{N}\mathbf{1}_{[0,z)}(x_{i})|^p dw dz\\ \leq& \mathbb{E}\Big(|z_{1}z_{2}\ldots z_{d}-
\frac{1}{N}\sum_{i=1}^{N}\mathbf{1}_{[0,z)}(x_{i})|^p\Big)\\ \leq& \frac{|\mathscr{I}|^{p/2}}{N^{p}}.
\end{align*}

Therefore, from \eqref{CNbd0}, we have,

\begin{equation}\label{exup1}
    \mathbb{E}(L_{P}^{p}(D_N,x))\leq\frac{d^{\frac{p}{2}}\cdot c_2(d)^{\frac{p}{2}}}{N^{\frac{p}{2}+\frac{p}{2d}}}.
\end{equation}

Besides, we have

$$
    |I(f)-\tilde{I}(f,\mathbf{P})|^{p}\leq (\sup_{f\in F_{d,q}^{*},\|f\|_{d,q}^{*}\leq 1 }|I(f)-\tilde{I}(f,\mathbf{P})|)^{p}.
$$

Combining with \eqref{soboappx1}, we have

\begin{equation}\label{supapplp}
\begin{aligned}
    \sup_{f\in F_{d,q}^{*},\|f\|_{d,q}^{*}\leq 1 }|I(f)-\tilde{I}(f,\mathbf{P})|^{p}&\leq (\sup_{f\in F_{d,q}^{*},\|f\|_{d,q}^{*}\leq 1 }|I(f)-\tilde{I}(f,\mathbf{P})|)^{p}\\&=L_{P}^{p}(D_N,x).
    \end{aligned}
\end{equation}

Combining with \eqref{exup1}, we complete the proof.
\end{proof}

\begin{rem}
Theorem \ref{exunibd1} gives an upper bound of $p-$moment of integral approximation error for functions in Sobolev space $ F_{d,q}^{*}$ using the estimation for expected $L^p-$discrepancy, where $\frac{1}{p}+\frac{1}{q}=1$. If $p=2$ in Theorem \ref{exunibd1}, we obtain mean square error bound of uniform integration approximation in Sobolev space $ F_{d,2}^{*}$, this conclusion is consistent with upper bound for Sobolev space $\mathcal{H}^{\mathbf{1}}(K)$ in Theorem \ref{epecintap}. Easy to see that $ F_{d,2}^{*}\subset \mathcal{H}^{\mathbf{1}}(K)$, Theorem \ref{exunibd1} is actually contained by Theorem \ref{epecintap} for the case of $p=2$. But by using expected $L_p-$discrepancy($p\ge 1$ and can be arbitrary) bounds in Theorem \ref{exunibd1}, we obtain upper bounds of $p-$moment of integral approximation error for functions in general Sobolev space $ F_{d,q}^{*}$, comparing with the reproducing kernel method in Theorem \ref{epecintap}. 
\end{rem}

\begin{cor}\label{isogridp}
For any $d \ge 2, 1\leq p<\infty$ and $N\in \mathbb{N}$, $d-$dimensional sampling set $y=\{y_{1}, y_{2}, y_{3}, \ldots, y_{N}\}$ is uniformly distributed in the subset $Q_{1}, Q_{2}, Q_{3}, \ldots,
Q_{N}$ which is isometric grid partition of $[0,1]^{d}$, then for functions $f$ in Sobolev space $ F_{d,q}^{*}$, where $\frac{1}{p}+\frac{1}{q}=1$, we have

\begin{equation}
\mathbb{E}\big[\sup_{f\in F_{d,q}^{*},\|f\|_{d,q}^{*}\leq 1 }\Big|\frac{1}{N}\sum_{n=1}^{N}f(y_n)-\int_{[0,1]^{d}}f(x)dx\Big|^p\big]\leq \frac{d^{\frac{p}{2}}}{N^{\frac{p}{2}+\frac{p}{2d}}}.
\end{equation}
\end{cor}

\begin{cor}\label{HSFCS1}
For any $d \ge 2, 1\leq p<\infty$ and $N\in \mathbb{N}$, $d-$dimensional sampling set $z=\{z_{1}, z_{2}, z_{3}, \ldots, z_{N}\}$ is uniformly distributed in the subset $E_{1}, E_{2}, E_{3},\ldots,
E_{N}$ which is a partition of $[0,1]^{d}$ formed by HSFC-based sampling, then for functions $f$ in Sobolev space $ F_{d,q}^{*}$, where  $\frac{1}{p}+\frac{1}{q}=1$, we have

\begin{equation}
\mathbb{E}\big[\sup_{f\in F_{d,q}^{*},\|f\|_{d,q}^{*}\leq 1 }\Big|\frac{1}{N}\sum_{n=1}^{N}f(z_n)-\int_{[0,1]^{d}}f(x)dx\Big|^p\big]\leq \frac{(2d\cdot\sqrt{d+3})^{\frac{p}{2}}}{N^{\frac{p}{2}+\frac{p}{2d}}}.
\end{equation}
\end{cor}

\begin{rem}
Corollary \ref{isogridp} and \ref{HSFCS1} give upper bound of $p-$moment of integral approximation error using isometric grid partition , this result is better than that of using HSFC-based sampling, we notice the sampling regime of $N=m^d$ for dimension $d$, which implies the sampling number is an exponential dependence, but the HSFC-based sampling with scrambled van der Corput inputs does not require the highly composite sample sizes that the grid sampling requires, particularly for large $d$. 
\end{rem}

\begin{cor}\label{slln}
For any $d \ge 2, 1\leq p<\infty,$ and $N\in \mathbb{N}$, $d-$dimensional sampling set $t=\{t_{1}, t_{2}, t_{3}, \ldots, t_{N}\}$ is uniformly distributed in the subset $\Omega_{1}, \Omega_{2}, \Omega_{3}, \ldots, \Omega_{N}$ which is some general partition of $[0,1]^{d}$, then for functions $f$ in Sobolev space $ F_{d,q}^{*}$, where $\frac{1}{p}+\frac{1}{q}=1$, then

\begin{equation}
    \mathbb{P}\Big(\lim_{N\rightarrow \infty}\frac{1}{N}\sum_{n=1}^{N}f(t_n)=\int_{[0,1]^{d}}f(x)dx\Big)=1.
\end{equation}
\end{cor}

\begin{rem}
\eqref{exup1}, \eqref{supapplp} and Markov’s inequality imply Corollary \ref{slln}, which provides a strong law of large numbers for integration on stratified random sampling formed by equal measure partition in Sobolev space $F_{d,q}^{*}$, this could be seen as an application of stratified sampling and $p-$moment of integral approximation error. The former result is for $L_p-$space and digital nets randomized by a nested uniform scramble, see \cite{OR2021}.
\end{rem}

\begin{cor}\label{pmrs}
Let $z=\{z_{1},z_{2},\ldots,z_{N}\}$ be simple random sampling points uniformly distributed in $[0,1]^{d}$, then we have

\begin{equation}
    \mathbb{E}(L_{P}^{p}(D_N,z))=O(\frac{1}{N^{\frac{p}{2}}}), 
\end{equation}
for $N\rightarrow\infty$, where $A=O(B)$ means the quantities $A,B$ on both sides of the sign $O$ are infinitesimals of the same order.
\end{cor}

\begin{proof}

Considering,

$$
\mathbb{E}(L_P^p(D_N,z))=\mathbb{E}
(\int_{[0,1]^d}|\frac{1}{N}
\sum_{n=1}^{N}\textbf{1}_{[0,x)}(z_n)-\lambda([0,x))|^pdx).
$$

Suppose the whole summation as a random variable which defines on a region we let it $\Omega_{N}$, besides we set the probability measure be $\mu$, thus,

\begin{equation}\label{elpp}
\begin{aligned}
\mathbb{E}(L_P^p(D_N,z))&=\int_{\Omega_{N}}
\int_{[0,1]^d}|\frac{1}{N}
\sum_{n=1}^{N}\textbf{1}_{[0,x)}(z_n)-\lambda([0,x))|^pdxd\mu\\&
=\int_{[0,1]^d}\int_{\Omega_{N}}|\frac{1}{N}
\sum_{n=1}^{N}\textbf{1}_{[0,x)}(z_n)-\lambda([0,x))|^pd\mu dx.
\end{aligned}
\end{equation}

Each single random variable either lands in $[0,x)$ or does not, which is just a Bernoulli trial
with probability $\lambda([0, x))$ and thus the entire expression follows a Binomial distribution, i.e., 

\begin{equation}\label{BDN}
    \sum_{n=1}^{N}\textbf{1}_{[0,x)}(z_n)\sim \mathcal{B}(N,\lambda([0,x)).
\end{equation}

Therefore, as $n\rightarrow\infty$, according to the central limit theorem, which is,

\begin{equation}\label{CLT}
    \mathcal{B}(n,p)=\mathcal{N}(np,np(1-p)).
\end{equation}

Applying \eqref{CLT} to the above equation \eqref{BDN}, we have,

$$
    \frac{\sqrt{N}}{\sqrt{\lambda([0,x))(1-\lambda([0,x)))}}[\frac{1}{N}
\sum_{n=1}^{N}\textbf{1}_{[0,x)}(z_n)-\lambda([0,x))]\sim \mathcal{N}(0,1).
$$

Thus,

\begin{equation}\label{ztfb}
    (\frac{\sqrt{N}}{\sqrt{\lambda([0,x))(1-\lambda([0,x)))}})^{p}|\frac{1}{N}
\sum_{n=1}^{N}\textbf{1}_{[0,x)}(z_n)-\lambda([0,x))|^{p}\sim |X|^{p},
\end{equation}
where $X$ is a random variable satisfying $X\sim \mathcal{N}(0,1)$.

From \eqref{ztfb}, when $N\rightarrow\infty$, we have

    \begin{align*}
    &\int_{\Omega_{N}}|\frac{1}{N}
\sum_{n=1}^{N}\textbf{1}_{[0,x)}(z_n)-\lambda([0,x))|^pd\mu\\&= (\frac{\sqrt{\lambda([0,x))(1-\lambda([0,x)))}}{\sqrt{N}})^{p}\int_{-\infty}^{\infty}|X|^{p}d\mathcal{N}(0,1)\\&=(\frac{\sqrt{\lambda([0,x))(1-\lambda([0,x)))}}{\sqrt{N}})^{p}\frac{2^{\frac{p}{2}}}{\sqrt{\pi}}\Gamma(\frac{1+p}{2}).
    \end{align*}

Therefore,

\begin{equation}\label{gamn}
\begin{aligned}
    &\int_{[0,1]^{d}}\int_{\Omega_{N}}|\frac{1}{N}
\sum_{n=1}^{N}\textbf{1}_{[0,x)}(z_n)-\lambda([0,x))|^pd\mu dx\\&=\frac{2^{\frac{p}{2}}}{\sqrt{\pi}}\Gamma(\frac{1+p}{2})(\frac{1}{\sqrt{N}})^{p}\int_{[0,1)^{d}} \sqrt{\lambda([0,x))(1-\lambda([0,x)))}^{p}dx.
\end{aligned}
\end{equation}

Due to $0\leq \lambda([0,x))\leq 1$,

\begin{equation}\label{xp}
    \int_{[0,1)^{d}} \sqrt{\lambda([0,x))(1-\lambda([0,x)))}^{p}dx\leq (\frac{2}{2+p})^{d}
\end{equation}
holds.

Combining \eqref{elpp}, \eqref{gamn} and \eqref{xp}, we complete the proof.

\end{proof}

\begin{rem}
Corollary \ref{pmrs} gives the convergence order of expected $L_p-$discrepancy bound for simple random sampling. We follow the proof in \cite{SS2010} which gives convergence order of average $L_p-$discrepancy for simple random sampling. Comparing with \eqref{exup1}, convergence order of $p-$moment of $L_p-$discrepancy for stratified random sampling is better than that for simple random sampling.
\end{rem}

\section{Conclusion}\label{conclu}

We study the uniform integration approximation for stratified sampling formed by equal measure partition. The stratified samples could produce more uniform point distribution configuration than crude Monte Carlo sampling point set. We prove that stratified samples could obtain better convergence order of uniform integration approximation bounds in sense of randomness, comparing with the use of simple random samples and Latin hypercube samples in certain function space. In most function space composed of smooth enough functions, from Koksma-Hlawka inequality, good orders of expected approximation are obtained by formers. We deal with more general function space such as Sobolev space $\mathcal{H}^{\mathbf{1}}(K)$ and $F_{d,q}^{*}$ in this paper. For several special cases of equal measure partition, our approximation bounds are explicit, especially for HSFC-based sampling, we obtain upper bound of $p-$moment of integral approximation error under moderate sample size, which does not require the highly composite sample sizes that the jittered sampling requires.

\end{document}